\documentclass[12pt]{amsart}

\usepackage{amsfonts}
\usepackage{amsmath}
\usepackage{amssymb}
\usepackage{amstext}
\usepackage{amsthm}
\usepackage{mathtools}
\usepackage{amsopn}
\usepackage{mathrsfs}
\usepackage{geometry}
\usepackage{url}
\usepackage{amsrefs}
\usepackage{hyperref}
\usepackage[english]{babel}
\usepackage[autostyle, english = american]{csquotes}
\MakeOuterQuote{"}

\hypersetup{
	colorlinks=true,
	linkcolor=red,
	filecolor=magenta,      
	urlcolor=cyan,
}

\theoremstyle{definition}
\theoremstyle{remark}
\newtheorem{theorem}{Theorem}[section]
\newtheorem{lemma}[theorem]{Lemma}
\newtheorem{definition}[theorem]{Definition}
\newtheorem{corollary}[theorem]{Corollary}
\newtheorem{proposition}[theorem]{Proposition}
\newtheorem{remark}[theorem]{Remark}
\newtheorem{example}{Example}[section]

\numberwithin{equation}{section}

\DeclareMathOperator{\LT}{LT}
\DeclareMathOperator{\LM}{LM}

\DeclareMathOperator{\gen}{gen}

\DeclareMathOperator{\red}{red}
\DeclareMathOperator{\LC}{LC}

\begin{document}

\title{Multi-Rees algebras on principal ideal rings}

\author{Babak Jabbar Nezhad}
\address{Istanbul, Turkey}
\email{babak.jab@gmail.com}


\date{}
\thanks{Babak Jabbar Nezhad has also published under the name Babak Jabarnejad~\cite{jabarnejad2016rees}}

\dedicatory{}

\begin{abstract}
When $R$ is a Noetherian ring and we have a family of ideals in which every ideal contains at least one nonzero divisor, then it is already known that the defining ideal of the multi-Rees algebra of these ideals is equal to a saturated ideal. In such a case to get the defining ideal of the multi-Rees algebra we only need to saturate the first syzygies of direct sum of this family of ideals. However, this fact is not true, when at least one of these ideals does not contain any nonzero divisor. In this paper we show that the defining ideal of the multi-Rees algebra of a family of ideals of a polynomial ring over a principal ideal ring, is equal to another kind of saturated ideal, where we saturate more explicit polynomials other than just first syzygies. Please notice that in general some of these ideals may not contain nonzero divisors. Given this explicit formula, we can compute Gr\"{o}bner basis of the defining ideal using an elimination order, which we talk about in the present paper.
\end{abstract} 

\maketitle

\section{Introduction}

The concept of Gr\"{o}bner bases for polynomial rings over a field was presented by Buchberger~\cite{buchberger1965algorithmus}. He also gave generalizations of this concept over some rings (e.g.~\cite{buchberger1984critical}). During decades this concept got many applications in different areas, including applied mathematics, pure mathematics and engineering. Besides,  Trinks~\cite{t78}, gave a natural generalization of Gr\"{o}bner bases concept on polynomial rings over Noetherian rings based on $S$-polynomial concept. The concept of Gr\"{o}bner bases on polynomial rings over arbitrary rings is an active area in commutative algebra as it has many applications. See~\cite{kan88}, \cite{adams1994introduction}, \cite{yengui2006dynamical}, \cite{kacem2010dynamical}, \cite{gamanda2019syzygy}, \cite{EH2021}.

The concept of Rees algebra $R[It]$ is an important topic in commutative algebra as it encodes asymptotic behavior of the ideal $I$. Let $s_1,\dots,s_n$ be generators of the ideal $I$. We define the epimorphism $\phi$ from the polynomial ring $S=R[T_1,\dots,T_n]$ to the Rees algebra $R[It]$ by sending $T_i$ to $s_it$. Then $R[It]\cong S/\ker(\phi)$. The generating set of $\ker(\phi)$ (resp. $\ker(\phi)$) is referred to as the defining equations (resp. the defining ideal) of the Rees algebra $R[It]$. One can generalize these concepts and define the multi-Rees algebra $R[I_1t_1,\dots,I_rt_r]$. Then we will have a similar $\phi$ and a similar defining ideal. Again, the generating set of $\ker(\phi)$ is referred to as the defining equations of the multi-Rees algebra $R[I_1t_1,\dots,I_rt_r]$. Actually, the concept of the multi-Rees algebra $R[I_1t_1,\dots,I_rt_r]$ is the same as the Rees algebra of $R$-module $I_1\oplus\dots\oplus I_r$, denoted by $\mathcal{R}_R(I_1\oplus\dots\oplus I_r)$. As we know from~\cite[Theorem 1.4]{eisenbud2003rees} that if ideals $I_i$ in the commutative ring $R$ are finitely generated, then we have
$$
\mathcal{R}_R(I_1\oplus\dots\oplus I_r)\cong R[I_1t_1,\dots,I_rt_r].
$$

Some recent works about the Rees algebra of an ideal and the multi-Rees algebra of ideals include~\cite{VW91}, \cite{MU96}, \cite{MP13}, \cite{KLU2017}, \cite{R99}, \cite{LP14}, \cite{Sosa14}, \cite{jabarnejad2016rees}, \cite{coxlinsosa}.

When we have a polynomial ring over a field $k$ as $R=k[x_1,\dots,x_n]$ and $I_1,\dots,I_r$ are ideals of $R$, then using $\textit{MACAULAY2}$ one may compute the defining equations of the multi-Rees algebra $R[I_1t_1,\dots,I_rt_r]$. Actually in $\textit{MACAULAY2}$ this can be done for some specific fields. The Rees algebra package is explained in~\cite{eisenbud2018reesalgebra}. Here we briefly explain the algorithm. Let $R$ be a Noetherian ring. Suppose that the $R$-module $M$ is generated by $s_1,\dots,s_m$. Then we consider the presentation matrix $\Phi_{m\times l}$ of $M$ associated to the generating set $\{s_1,\dots,s_m\}$. We consider the row matrix $T=[T_1,T_2,\dots,T_m]$. By $I_1(T.\Phi)$ we mean the ideal generated by entries of the matrix $T.\Phi$. Now if $M$ contains an element $f$ which is a nonzero divisor, and  $M[f^{-1}]$ is a free module, then the defining ideal of the Rees algebra $\mathcal{R}_R(M)$ is equal to    
$$
I_1(T.\Phi):f^\infty.
$$

Now we convert this problem to the multi-Rees algebra $R[I_1t_1,\dots,I_rt_r]$. Suppose $I_1,\dots,I_r$ are ideals of a Noetherian ring $R$. We want to compute the defining equations of the Rees algebra $\mathcal{R}_R(I_1\oplus\dots\oplus I_r)$. Assume that $\{s_1,\dots,s_m\}$ is a generating set of the $R$-module $I_1\oplus\dots\oplus I_r$. Also, $\Phi_{m\times l}$ is the presentation matrix of $I_1\oplus\dots\oplus I_r$ which is associated to the generating set $\{s_1,\dots,s_m\}$. If we consider the row matrix $T=[T_1,T_2,\dots,T_m]$, then the entries of the matrix $T.\Phi$ in some sense are first syzygies of the $R$-module $I_1\oplus\dots\oplus I_r$. Now if every ideal $I_i$ contains an element $f_i$ which is a nonzero divisor, then $(I_1\oplus\dots\oplus I_r)[(f_1 ... f_r)^{-1}]$ is a free module. As it is explained above, under this condition (i.e. $f_i$ are nonzero divisors), we see that the defining ideal of the multi-Rees algebra is equal to    
$$
I_1(T.\Phi):(f_1 ... f_r)^\infty.
$$
In other words, if we saturate the ideal of first syzygies of $I_1\oplus\dots I_r$ by powers of an specific element which is a nonzero divisor, then we get the defining ideal of the multi-Rees algebra $R[I_1t_1,\dots,I_rt_r]$. Then using the concept of Gr\"{o}bner bases in rings that these bases are computable, we may compute the defining equations. The papers~\cite{coxlinsosa}, \cite{KSJ2019}, are in the same direction. In~\cite{coxlinsosa}, authors discuss the multi-Rees algebra $k[x_1,\dots,x_n][I_1t_1,\dots,I_rt_r]$ ($k$ is a field), where $I_i$ are monomial ideals of $k[x_1,\dots,x_n]$. In~\cite{KSJ2019}, authors argue the multi-Rees algebra $k[x_1,\dots,x_n][I_1t_1,\dots,I_rt_r]$ ($k$ is a field), where $I_i$ are arbitrary ideals of $k[x_1,\dots,x_n]$. As wee see in all mentioned papers all ideals contain nonzero divisors and to get the defining ideal of the multi-Rees algebra they saturate either all first syzygies or some of first syzygies of the module $I_1\oplus\dots\oplus I_r$. However when we discuss the multi-Rees algebra $R[x_1,\dots,x_n][I_1t_1,\dots,I_rt_r]$ for a Noetherian ring $R$ ($I_i$ are ideals of $R[x_1,\dots,x_n]$) and some of $I_i$ do not contain any nonzero divisors we cannot get the defining ideal by saturating first syzygies of the $R[x_1,\dots,x_n]$-module $I_1\oplus\dots\oplus I_r$. In some sense, this is the subject of our paper. In fact, in such a case to get the defining ideal we need to saturate more polynomials which are not just first syzygies and actually some of them are not linear in $T$. See Examples~\ref{counter1}, \ref{counter2}, \ref{counter3}. Actually in Example~\ref{counter3}, we prove this claim, where the ground ring is a principal ideal ring and we have zero divisors and the ideal does not contain any nonzero divisors. Please note that the first syzygies of a finitely generated module over a polynomial ring over a principal ideal ring is given in~\cite{gamanda2019syzygy}, \cite{BJGR}. For example to observe Example~\ref{counter3}, one may look at~\cite[Section 5]{BJGR}, to see how syzygies look like. 

In the present paper we introduce a formula for the defining ideal of the multi-Rees algebra of ideals for certain rings when some of ideals may not contain nonzero divisors. Actually we show that the defining ideal is still equal to a saturated ideal but we saturate more polynomials and not just first syzygies. We develop this formula for the multi-Rees algebra $R[x_1,\dots,x_n][I_1t_1,\dots,I_rt_r]$, where $R$ is a principal ideal ring, and $I_j$ are ideals of $R[x_1,\dots,x_n]$. By~\cite[Lemma 10, Corollary 11]{hungerford1968structure}, every principal ideal ring is a finite direct sum of quotients of Principal Ideal Domains (PIDs). Therefore, every principal ideal ring is isomorphic to $\oplus_{i=1}^{e}R_i/N_iR_i$, where $R_i$ is a principal ideal domain and $N_i$ is either zero or $N_i=p_i^{n_i}$, where $p_i$ is a prime element in $R_i$. Then the multi-Rees algebra $R[x_1,\dots,x_n][I_1t_1,\dots,I_rt_r]$ is isomorphic to the finite direct sum of multi-Rees algebras of the corresponding ideals on $R_i/N_iR_i[x_1,\dots,x_n]$. As a result, it is enough to build this formula for multi-Rees algebras on corresponding ideals in $R_i/N_iR_i[x_1,\dots,x_n]$. 

Therefore, this paper contains two main sections, which are Sections~\ref{pir-section} and \ref{pid-section}. In Section~\ref{pir-section}, which concludes the main contribution of this paper, we give a formula for the defining ideal of the multi-Rees algebra $R/p^mR[x_1,\dots,x_n][I_1t_1,\dots,I_rt_r]$, where $R$ is a PID, $p$ is a prime element in $R$, $m\ge 2$ and $I_j$ are ideals of $R/p^mR[x_1,\dots,x_n]$. As we see in this case some of ideals may not contain nonzero divisors. So that to get the defining ideal of the multi-Rees algebra we saturate some explicit polynomials, in which some of them are not among first syzygies of the module $I_1\oplus\dots\oplus I_r$. As we said before, in Example~\ref{counter3}, we show that if we only saturate first syzygies, then we cannot obtain the defining ideal of the multi-Rees algebra. Please notice that if $m=1$, then $R/pR$ is a field. In this section, we introduce an algorithm and using it we show that in this case the defining ideal is equal to a saturated ideal of some explicit equations in $T$, where some of them are not linear in $T$. 

In Section~\ref{pid-section}, which is a minor contribution of this paper, we give a formula for the multi-Rees algebra $R[x_1,\dots,x_n][I_1t_1,\dots,I_rt_r]$, where $R$ is a PID and $I_j$ are ideals of $R[x_1,\dots,x_n]$. To do this part, again we introduce another algorithm. One may wonder why we are doing this case while, according to the explained method in~\cite{eisenbud2018reesalgebra}, the defining ideal is equal to
$$
I_1(T.\Phi):(f_1 ... f_r)^\infty\quad \text{"explained above"},
$$
and the entire $I_1(T.\Phi)$ is saturated, which is the ideal generated by all first syzygies. To address this concern, we show that in this minor section of the paper, we don't need to saturate all first syzygies. Instead, we only need to saturate some of the Koszul relations. 

Both mentioned algorithms in Sections~\ref{pir-section} and \ref{pid-section}, are modifications of the division algorithm, where we don't put order on all variables. Finally, using elimination order and Gr\"{o}bner bases concepts on polynomial rings over principal ideal rings, we can compute defining equations of the multi-Rees algebra. Actually, in this paper we explain how to use elimination order to compute Gr\"{o}bner basis of the defining ideal. Therefore, one may see this paper as an application of Gr\"{o}bner bases concept on polynomial rings over principal ideal rings. Please be notified that the concept of Gr\"{o}bner bases on polynomial rings over principal ideal rings is discussed in~\cite{gamanda2019syzygy}, \cite{BJGR}. 
 
At the end, as an special case, we consider the multi-Rees algebra 
$$
R/p^mR[x_1,\dots,x_n][I_1t_1,\dots,I_rt_r],
$$ 
where $R$ is a principal ideal domain, $p$ is a prime element of $R$, and $I_j$ are term ideals of $R/p^mR[x_1,\dots,x_n]$ (ideals that are generated by polynomials with one term). In this special case we show that the defining ideal is equal to
$$
\mathcal{L}:(x_1 ... x_n)^\infty,
$$
where the ideal $\mathcal{L}$ is generated by some explicit polynomials.

Similarly, as an special case, we consider the multi-Rees algebra 
$$
R[x_1,\dots,x_n][I_1t_1,\dots,I_rt_r],
$$
where $R$ is a PID, and $I_j$ are term ideals of $R[x_1,\dots,x_n]$. In this special case we show that the defining ideal is equal to 
$$
\mathcal{L}:(p_1 ... p_lx_1 ... x_n)^\infty,
$$
where the ideal $\mathcal{L}$ is generated by some explicit polynomials. Also, $p_i$ are some distinct prime factors of coefficients of generators of $I_j$.

Also, we give some examples, in which we compute defining equations of the multi-Rees algebra.

\section{Elimination order}

In this section we state elimination order concept for the ring $A[\mathbf{x}]$, where $A$ is a Noetherian ring and $\mathbf{x}=x_1,\dots,x_n$. This fact is a natural generalization of field case which could be seen easily. But to our knowledge it is not officially stated in the literature, then we state it here.

\begin{definition}
Let $I$ be an ideal $A[\mathbf{x}]$ which has a monomial order. Suppose $\{g_1,...,g_m\}\subseteq I$. We say $\{g_1,...,g_m\}$ is a Gr\"{o}bner basis for $I$ if $\langle \LT(g_1),...,\LT(g_m)\rangle=\LT(I)$, where $\LT(I)$ is the ideal generated by the leading terms of elements of $I$. 
\end{definition}

From~\cite[Corollary 4.1.15]{adams1994introduction}, we see that if $I$ is an ideal of $A[\mathbf{x}]$ and $G=\{g_1,...,g_m\}$ is a Gr\"{o}bner basis for $I$, then $G$ generates $I$. The concept of elimination order for $A[\mathbf{x}]$ is similar to the field case. We have similar result to \cite[Theorem 3.3]{ene2011gr} with a similar proof.

\begin{lemma}\label{elimination}
Let $I\subseteq A[\mathbf{x}]$ be an ideal and $1\le t\le n$ an integer. If $G$ is a Gr\"{o}bner basis of $I$ with respect to some elimination order $\prec$ for $x_1,\dots,x_t$, then $G_t=G\cap A[x_{t+1},\dots,x_n]$ is a Gr\"{o}bner basis of $I_t=I\cap A[x_{t+1},\dots,x_n]$ with respect to the induced order on the subring $A[x_{t+1},\dots,x_n]$.
\end{lemma}
\begin{proof}
Let $G=\{g_1,\dots,g_m\}$ and $G_t=\{g_1,\dots,g_s\}$. We show that $\langle\LT(g_1),\dots,\LT(g_s)\rangle=\LT(I_t)$. First of all, by the choice of monomial order $\LT(g_j)\notin A[x_{t+1},\dots,x_n]$ for all $j>s$. Let $f\in I_t$ be a nonzero polynomial. Then $\LT(f)=\sum_{j=1}^{m}r_jh_j\LT(g_j)$, where $r_j\in R$ and $h_j$ are monomials. But variables of $\LT(f)$ are between $x_{t+1},\dots,x_n$. Hence, $\LT(f)=\sum_{j=1}^{s}r_jh_j\LT(g_j)$.
\end{proof}

\begin{definition}
A ring $A$ is solvable when given $a,a_1,\dots,a_m\in A$, we can determine whether $a\in\langle a_1,\dots,a_m\rangle$ and if it is, then we can find $b_1,\dots,b_m\in A$ such that $a=a_1b_1+\dots+a_mb_m$. 
\end{definition}

\begin{remark}
If $A$ is solvable, then $A/\alpha A$ ($\alpha\in A$) is solvable. Because determining whether $a+\alpha A\in\langle a_1+\alpha A,\dots,a_k+\alpha A\rangle$ and finding solutions are equivalent to determining whether $a\in\langle a_1,\dots,a_k,\alpha\rangle$ and finding solutions.  
\end{remark}

Let $A$ be a solvable PID. Then Gr\"{o}bner basis on polynomial rings over $A/\alpha A$ is computable~\cite{gamanda2019syzygy}. Note that every Euclidean domain is solvable. Please notice that in~\cite{gamanda2019syzygy}, authors consider discrete rings, which means rings whose equality is decidable. But whenever it is possible we disregard this condition, as we are not dealing with constructive mathematics.

\section{Defining ideal of the multi-Rees algebra when the ground ring is $R/p^mR$}\label{pir-section}

In this section we fix ideals $I_1,\dots,I_r$ of $R/p^mR[\mathbf{x}]$, where $\mathbf{x}=x_1,\dots,x_n$, $R$ is a PID, $p$ is a prime element of $R$ and $m\ge 2$. Please notice that when $m=1$, the ground ring is a field. We denote a fixed generating set of the ideal $I_j$ by $\gen(I_j)$. Let $\{f_1,\dots,f_q\}=\bigcup_{i=1}^{r}\gen(I_i)$.

In $R/p^mR$, every nonzero element can be shown uniquely as $vp^n+p^mR$, where $n<m$ and $p$ does not divide $v$. Because if $up^{n_1}+p^mR=vp^{n_2}+p^mR$, then $up^{n_1}-vp^{n_2}=lp^{m}$. Without loss of generality we may assume $n_1<n_2$. Hence $u-vp^{n_2-n_1}=lp^{n_3}$, $n_3>0$. Then $p$ divides $u$, which is a contradiction. Clearly $up^{n_1}+p^mR$ divides $vp^{n_2}+p^mR$ iff $n_1\le n_2$. In such a case we write 
$$
\frac{vp^{n_2}+p^mR}{up^{n_1}+p^mR}=vu^{-1}p^{n_2-n_1}+p^mR.
$$

\begin{proposition}\label{nonzero-divisor}
Let $f\in R/p^mR[\mathbf{x}]$. Suppose at least one term of $f$ is a nonzero divisor and $0\neq g\in R/p^mR[\mathbf{x}]$. If $p^n+p^mR\mid g$ and $n$ is maximum, then $p^{n+1}+p^mR\nmid fg$.
\end{proposition}
\begin{proof}
Let $f=(r_1+p^mR)f_1+\dots+(r_k+p^mR)f_k+\dots+(r_t+p^mR)f_t$, where $f_j$ are monomials. Without loss of generality we may assume $r_1+p^mR,\dots,r_k+p^mR$ are nonzero divisors and $r_{k+1}+p^mR,\dots,r_t+p^mR$ are zero divisors. Note that by the assumption $k\ge 1$. We see that $p\mid r_j$ for $k+1\le j\le t$. Let $g=(s_1+p^mR)g_1+\dots+(s_l+p^mR)g_l$ ($g_j$ are monomials) and $p^{n+1}+p^mR\mid fg$. We have $p^n\mid s_j$ for $1\le j\le l$ and $n$ is maximum. Let $\overline{f}=r_1f_1+\dots+r_tf_t$ and $\overline{g}=s_1g_1+\dots+s_lg_l$. Then we have $\overline{g}=p^ng^{'}$ and $n$ is maximum. We have
\begin{gather*}
((r_1f_1+\dots+r_kf_k)+(r_{k+1}f_{k+1}+\dots+r_tf_t))\overline{g}-p^{n+1}a=p^{m}b\\
\Rightarrow ((r_1f_1+\dots+r_kf_k)+(r_{k+1}f_{k+1}+\dots+r_tf_t))\overline{g}=p^{n+1}h\\
\Rightarrow ((r_1f_1+\dots+r_kf_k)+(r_{k+1}f_{k+1}+\dots+r_tf_t))g^{'}=ph\\
\Rightarrow (r_1f_1+\dots+r_kf_k)g^{'}+pa=ph\Rightarrow p\mid g^{'},
\end{gather*}
which is a contradiction.
\end{proof}

\begin{corollary}\label{nonzero-divisor1}
Let $f\in R/p^mR[\mathbf{x}]$. If at least one term of $f$ is a nonzero divisor, then $f$ is a nonzero divisor.
\end{corollary}

Let 
$$
S:=R/p^mR[\mathbf{x}]\left[\{T_{k,j}\}_{1\le j\le r,f_k\in\gen(I_j)}\right].
$$ 
We define and fix the $R/p^mR[\mathbf{x}]$-algebra epimorphism 
$$
\phi : S \to R/p^mR[\mathbf{x}][I_1t_1,\dots, I_rt_r], \ \text{by} \ \phi(T_{k,j})=f_kt_j,
$$
where $R/p^mR[\mathbf{x}][I_1t_1,\dots, I_rt_r]$ is the multi-Rees algebra of ideals $I_j$. Let $\mathscr{L}=\ker(\phi)$. In the rest of this section $\mathbf{T}$ is the set of defined $T_{k,j}$.

\begin{definition}
If $f\in R/p^mR[\mathbf{x},\mathbf{T}]$, then we say $f$ is reduced if the greatest common divisor of its coefficients is 1. 
\end{definition}

\begin{definition}
For every $f\in R/p^mR[\mathbf{x},\mathbf{T}]$, we factor out the greatest common divisor of its coefficients and we denote the remaining polynomial by $f_{\red}$. We also denote the greatest common divisor of coefficients of $f$ by $f_{\gcd}$. Please note that $f=f_{\gcd}f_{\red}$, and $f_{\red}$ and $f_{\gcd}$ are unique up to unit.
\end{definition} 

In this section, we fix the lexicographic order $\prec_1$ in $R/p^mR[\mathbf{x},\mathbf{T}]$ on $T$ monomials as follows: If $j_1<j_2$, then $T_{k_1,j_1}\prec_1 T_{k_2,j_2}$. For any $j$ we say $T_{k_1,j}\prec_1 T_{k_2,j}$, if we have one of the following conditions: 1. $f_{{k_1}_{\gcd}}\mid f_{{k_2}_{\gcd}}$ and $f_{{k_2}_{\gcd}}\nmid f_{{k_1}_{\gcd}}$. 2. If $f_{{k_1}_{\gcd}}\mid f_{{k_2}_{\gcd}}$, $f_{{k_2}_{\gcd}}\mid f_{{k_1}_{\gcd}}$, then $k_1<k_2$. Note that when we put terms of a polynomial in descending order, coefficients are polynomials in $x$. Note that since we have not put any order on $x$ then leading coefficient of $f\in R/p^mR[\mathbf{x},\mathbf{T}]$ (denoted by $\LC(f)$) is a polynomial in $x$. Finally, we denote the leading monomial of $f$ by $\LM(f)$, which is a monomial in $T$. Clearly the leading term of $f$ (denoted by $\LT(f)$) is equal to $\LC(f)\LM(f)$. 

\begin{theorem}\label{division-algorithm-2}
Let $F=(f_1,\dots,f_t)$ be an ordered $t$-tuple of polynomials in $R/p^mR[\mathbf{x},\mathbf{T}]$, and let $f$ be a polynomial in $R/p^mR[\mathbf{x},\mathbf{T}]$. Consider the following algorithm:
	
1. $g_1=g_2=\dots=g_t=0$, $s=0$, $g=f$, and $h=f$.
	
2. Find the smallest $j\in\{1,\dots,t\}$ such that $\LC(g)_{\gcd}\LM(g)$ is a multiple of $\LC(f_j)_{\gcd}\LM(f_j)$. If such a $j$ exists, replace $g_j$ by $\LC(f_j)_{\red}\left(g_j+\frac{\LT(g)}{\LT(f_j)}\right)$, replace other $g_k$ by $\LC(f_j)_{\red}g_k$. Finally, replace $g$ by 
$$
\LC(f_j)_{\red}\left(g-\frac{\LT(g)}{\LT(f_j)}f_j\right),
$$ 
and $s$ by $\LC(f_j)_{\red}s$. Also, replace $h$ by $\LC(f_j)_{\red}h$.
	
3. Repeat step 2 until there is no more $j\in\{1,\dots,t\}$ such that $\LC(g)_{\gcd}\LM(g)$ is a multiple of $\LC(f_j)_{\gcd}\LM(f_j)$. Then replace $s$ by $s+\LT(g)$ and $g$ by $g-\LT(g)$.
	
4. If now $g\neq 0$, start again with step 2. If $g=0$ stop.
	
This is an algorithm which gives us $t$-tuple $(g_1,\dots,g_t)\in R[\mathbf{x},\mathbf{T}]^t$ and the polynomial $s\in R[\mathbf{x},\mathbf{T}]$ such that
$$
af=h=g_1f_1+\dots+g_tf_t+s,
$$
and such that the following conditions are satisfied.
	
a. If $s\neq 0$, then none of terms of $s$ is divisible by any of $\LC(f_1)_{\gcd}\LM(f_1),\dots,\LC(f_t)_{\gcd}\LM(f_t)$.
	
b. $a$ is a product of $\LC(f_j)_{\red}$.  
\end{theorem}
\begin{proof}
We see that the algorithm terminates in finitely many steps, because in each step the leading term becomes strictly smaller. For divisibility, since the leading term of $g$ in each step becomes strictly smaller, when we replace $s$ by $s+\LT(g)$, $\LT(g)$ does not combine with other terms of $s$. So that in step 3 the power of $p+p^mR$ is not increased in terms of $s$. On the other hand, by induction we assume that in a step none of terms of $s$ is divisible by any of $\LC(f_1)_{\gcd}\LM(f_1),\dots,\LC(f_t)_{\gcd}\LM(f_t)$. When we replace $s$ by $\LC(f_j)_{\red}s$, if a term of $\LC(f_j)_{\red}s$ is divisible by, say, $\LC(f_1)_{\gcd}\LM(f_1)$, then corresponding term in $s$, say, $s_l$ is divisible by $\LM(f_1)$, so that it is not divisible by $\LC(f_1)_{\gcd}$. Let $p^n+p^mR$ be the greatest factor of $s_l$. By Proposition~\ref{nonzero-divisor}, $p^{n+1}+p^mR\nmid\LC(f_j)_{\red}s_l$. Hence $\LC(f_1)_{\gcd}\nmid\LC(f_j)_{\red}s_l$, which is a contradiction.  
\end{proof}

Now we want to show that the defining ideal (i.e. $\ker(\phi)=\mathscr{L}$) of the multi-Rees algebra $R/p^mR[\mathbf{x}][I_1t_1,\dots, I_rt_r]$ is a saturated ideal of some explicit equations. But as we said in the introduction these explicit equations are not just first syzygies of the $R/p^mR[\mathbf{x}]$-module $I_1\oplus\dots\oplus I_r$. We fix the set

\begin{gather*}
H=\{p^n+p^mRT_{{k_1},{j_1}}^{n_1}\dots T_{{k_s},{j_s}}^{n_s}; f_{{k_t}_{\gcd}}\ \text{is not a unit}, 1\le t\le s,\\
\forall T_{k,{j_t}}, (k\neq k_t), T_{k_t,j_t}\prec_1 T_{k,j_t}, 0\le n,n_t, n+n_1+\dots+n_s=m\}.\\
\end{gather*}
It is clear that $p^n+p^mRT_{{k_1},{j_1}}^{n_1}\dots T_{{k_s},{j_s}}^{n_s}\in\mathscr{L}$.

Let $p^\alpha+p^mRT_{{k_1},{j_1}}^{\beta_1}\dots T_{{k_s},{j_s}}^{\beta_s}\neq 0\in\mathscr{L}$. Then clearly it is a multiple of an element of $H$. 

In this section by $\langle\dots\rangle$ we mean the ideal generated by elements inside the brackets. In particular for a set $A$, by $\langle A\rangle$, we mean the ideal generated by all elements of the set $A$.

For every $j$, let $T_{k_j,j}$ be the smallest $T$ under $\prec_1$. Then, in this section, we fix the ideal 
\begin{gather*}
\mathcal{L}=\langle\{(f_{k_j}T_{k,j}-f_kT_{k_j,j})_{\red};1\le j\le r,f_k\in I_j\}\rangle+\langle H\rangle.
\end{gather*}

\begin{theorem}\label{equations2}
	If
	$$
	\mathcal{E}=\mathcal{L}:\left(\prod(f_{k_j})_{\red}\right)^\infty,
	$$
	then $\mathscr{L}=\mathcal{E}$.
\end{theorem}
\begin{proof}
	If $f\in\mathcal{E}$, then 
	\begin{gather*}
	\left(\prod(f_{k_j})_{\red}\right)^lf\in\mathcal{L}.
	\end{gather*}
	If we substitute $f_kt_j$ for $T_{k,j}$ we have
	$$
	\left(\prod(f_{k_j})_{\red}\right)^lf(f_kt_j)=0. 
	$$
	But 
	$$
	\left(\prod(f_{k_j})_{\red}\right)^l
	$$ 
	is a nonzero divisor. Then $f(f_kt_j)=0$. Hence $f\in\mathscr{L}$ . Now we prove the other direction.
	
	Let $h_i$ ($1\le i\le v$) be mentioned generators of $\mathcal{L}$ and let $f\in\mathscr{L}$. Then by Theorem~\ref{division-algorithm-2}, we have
	$$
	\left(\prod(f_{k_j})_{\red}\right)^lf=g_1h_1+\dots+g_vh_v+s.
	$$ 
	If $s\neq 0$, then none of its terms is divisible by any of $\LC(h_i)_{\gcd}\LM(h_i)$. Note that, with regard to the order that we have, terms of $s$ are in the form of $ab$, where $a$ is a polynomial in $R/p^mR[\mathbf{x}]$ and $b$ is a monomial in $T_{k,j}$. We see that when we substitute $f_kt_j$ for $T_{k,j}$ in equation above $s$ becomes zero.
	
	But what we have defined, for every $T_{k,j}$ we have $f_{{k_j}_{\gcd}}\mid f_{k_{\gcd}}$. Also, $\LC_{\gcd}(f_{k_j}T_{k,j}-f_kT_{{k_j},j})_{\red}$ is a unit. Hence only these $T_{k_j,j}$ appear in terms of $s$ between $T$ variables.
	
	Then terms of $s$ are in the form $p^v+p^mRg_{\mathbf{x}}T_{{k_1},1}^{m_1}\dots T_{{k_r},r}^{m_r}$, where $g_{\mathbf{x}}$ is a reduced polynomial in $R[\mathbf{x}]$, and it is a nonzero divisor. When we substitute $f_kt_j$ for $T_{k,j}$, these terms individually become zero. Hence by the argument above about $H$, $g_{\mathbf{x}}=0$. Then $s=0$.
\end{proof}

Now we state an elementary result which is well-known in the literature.

\begin{proposition}\label{equations-grobner1}
	Let 
	$$
	\mathcal{F}=\langle y\prod(f_{k_j})_{\red}-1\rangle+\mathcal{L}\subseteq R/p^mR[\mathbf{x},\mathbf{T},y],
	$$
	and
	$$
	\mathcal{E}=\mathcal{L}:\left(\prod(f_{k_j})_{\red}\right)^\infty.
	$$
	Then $\mathcal{E}=\mathcal{F}\cap R/p^mR[\mathbf{x},\mathbf{T}]$.
\end{proposition}
\begin{proof}
	If $f\in\mathcal{E}$, then 
	$$
	\left(\prod(f_{k_j})_{\red}\right)^lf\in\mathcal{L}.
	$$ 
	But
	\begin{gather*}
	1-\left(y\prod(f_{k_j})_{\red}\right)^l=\left(1-y\prod(f_{k_j})_{\red}\right)g\\
	\Rightarrow f=\left(y\prod(f_{k_j})_{\red}\right)^lf+\left(1-y\prod(f_{k_j})_{\red}\right)gf\in\mathcal{F}.
	\end{gather*}
	Then $\mathcal{E}\subseteq\mathcal{F}\cap R[\mathbf{x},\mathbf{T}]$. For the other direction, if $f\in\mathcal{F}\cap R/p^mR[\mathbf{x},\mathbf{T}]$, then
	\begin{gather*}
	f=a\left(1-y\prod(f_{k_j})_{\red}\right)+\sum b_\alpha c_\alpha,
	\end{gather*} 
	where $a,b_\alpha\in R/p^mR[\mathbf{x},\mathbf{T},y]$, and $c_\alpha$ are generators of $\mathcal{L}$. It is enough to replace $y$ by 
	$$
	\frac{1}{\prod(f_{k_j})_{\red}}
	$$ 
	in the equation above, which proceeds to the other direction. Note that in this part we do calculation in the localization of the ring with multiplicative set generated by $\prod(f_{k_j})_{\red}$. 
\end{proof}
\begin{corollary}
	Suppose 
	$$
	\mathcal{F}=\langle y\prod(f_{k_j})_{\red}-1\rangle+\mathcal{L}.
	$$
	Let $\prec$ be an elimination order for $y$ on the ring $R[\mathbf{x},\mathbf{T},y]$ and let $G$ be a Gr\"{o}bner basis for $\mathcal{F}$ with this order. Then $G\cap R/p^mR[\mathbf{x},\mathbf{T}]$ is a Gr\"{o}bner basis for $\mathscr{L}$.
\end{corollary}

By what we have argued so far we can state the following result. Just notice that if $f$ is a term (a term is a polynomial with one term), then $f_{\red}$ is a term whose coefficient is a unit.

\begin{corollary}\label{equations4}
	Let all ideals $I_j$ be generated by terms. If 
	$$
	\mathcal{E}=\mathcal{L}:(x_1 ... x_n)^\infty,
	$$ 
	then $\mathscr{L}=\mathcal{E}$.
\end{corollary}

Now we go over some examples. In the following examples, whenever we compute Gr\"{o}bner basis, we use the following algorithms in~\cite{gamanda2019syzygy}: Division Algorithm 2.3, $S$-Polynomial Algorithm 2.4, and Buchberger's Algorithm 2.5. For convenience in each example instead of $a+N\mathbb{Z}$, we write $a$. We just remind that the meaning of the ideal $\mathcal{F}$, is mentioned in Proposition~\ref{equations-grobner1}.

\begin{example}\label{counter1}
	We consider the polynomial ring $\mathbb{Z}/9\mathbb{Z}[x_1,x_2,x_3]$. Let $f_1=2x_1^2x_2+6x_3, f_2=6x_1x_3, f_3=3x_3^2$. We consider the ideals $I_1=\langle f_1,f_2,f_3\rangle$ and $I_2=\langle f_2,f_3\rangle$. Then we have the epimorphism
	\begin{gather*}
	\phi:\mathbb{Z}/9\mathbb{Z}[T_{3,2},T_{2,2},T_{3,1},T_{2,1},T_{1,1},x_1,x_2,x_3]\rightarrow \mathbb{Z}/9\mathbb{Z}[x_1,x_2,x_3][I_1t_1,I_2t_2]\\
	T_{k,j}\mapsto f_kt_j.
	\end{gather*} 
	Suppose $\mathscr{L}=\ker(\phi)$. We have
	\begin{gather*}
	\mathcal{F}=\langle y(2x_1^2x_2+6x_3)(2x_1x_3)-1, (2x_1^2x_2+6x_3)T_{2,1}-6x_1x_3T_{1,1},(2x_1^2x_2+6x_3)T_{3,1}-3x_3^2T_{1,1},\\
	2x_1x_3T_{3,2}-x_3^2T_{2,2},3T_{2,2},T_{2,2}^2\rangle.
	\end{gather*}
	We put lexicographic order on the ring
	$$
	\mathbb{Z}/9\mathbb{Z}[y,T_{3,2},T_{2,2},T_{3,1},T_{2,1},T_{1,1},x_1,x_2,x_3].
	$$
	Let $G$ be intersection of Gr\"{o}bner basis of $\mathcal{F}$ with 
	$$
	\mathbb{Z}/9\mathbb{Z}[T_{3,2},T_{2,2},T_{3,1},T_{2,1},T_{1,1},x_1,x_2,x_3].
	$$
	Then we have
	\begin{gather*}
	\mathscr{L}=G=\langle x_1x_2T_{2,1}-3x_3T_{1,1}, 2x_1T_{3,1}-x_3T_{2,1}, 2x_1T_{3,2}-x_3T_{2,2}, 3T_{2,1}, T_{2,1}^2, 3T_{3,1}, T_{3,1}^2,\\
	3T_{2,2}, T_{2,2}^2, 3T_{3,2}, T_{3,2}^2, T_{2,1}T_{3,1}, T_{2,1}T_{2,2}, T_{2,1}T_{3,2}, T_{3,1}T_{2,2}, T_{3,1}T_{3,2}, T_{2,2}T_{3,2}\rangle.
	\end{gather*}
\end{example}

\begin{example}\label{counter2}
	We consider the polynomial ring $\mathbb{Z}/8\mathbb{Z}[x_1,x_2,x_3]$. Let $f_1=2x_1^2x_2, f_2=2x_1x_3, f_3=x_1^2, f_4=x_1^2x_2, f_5=x_1x_3$. We consider the ideals $I_1=\langle f_1,f_2\rangle$, $I_2=\langle f_1,f_3\rangle$ and $I_3=\langle f_3,f_4,f_5\rangle$ . Then we have the epimorphism
	\begin{gather*}
	\phi:\mathbb{Z}/8\mathbb{Z}[T_{5,3},T_{4,3},T_{3,3},T_{1,2},T_{3,2},T_{2,1},T_{1,1},x_1,x_2,x_3]\rightarrow \mathbb{Z}/8\mathbb{Z}[x_1,x_2,x_3][I_1t_1,I_2t_2,I_3t_3]\\
	T_{k,j}\mapsto f_kt_j.
	\end{gather*} 
	Suppose $\mathscr{L}=\ker(\phi)$. We have
	\begin{gather*}
	\mathcal{F}=\langle yx_1x_2x_3-1, x_1^2x_2T_{2,1}-x_1x_3T_{1,1}, x_1^2T_{1,2}-2x_1^2x_2T_{3,2},4T_{1,1},2T_{1,1}^2,T_{1,1}^3,\\
	x_1^2x_2T_{5,3}-x_1x_3T_{4,3}, x_1^2T_{5,3}-x_1x_3T_{3,3}\rangle.
	\end{gather*}
	We put lexicographic order on the ring
	$$
	\mathbb{Z}/8\mathbb{Z}[y,T_{5,3},T_{4,3},T_{3,3},T_{1,2},T_{3,2},T_{2,1},T_{1,1},x_1,x_2,x_3].
	$$
	Let $G$ be intersection of Gr\"{o}bner basis of $\mathcal{F}$ with 
	$$
	\mathbb{Z}/8\mathbb{Z}[T_{5,3},T_{4,3},T_{3,3},T_{1,2},T_{3,2},T_{2,1},T_{1,1},x_1,x_2,x_3].
	$$
	Then we have
	\begin{gather*}
	\mathscr{L}=G=\langle T_{1,2}-2x_2T_{3,2}, T_{4,3}-x_2T_{3,3}, x_1T_{5,3}-x_3T_{3,3}, T_{5,3}T_{1,1}-x_2T_{3,3}T_{2,1}, x_1x_2T_{2,1}-x_3T_{1,1},\\
	4T_{1,1}, 2T_{1,1}^2, T_{1,1}^3, 4T_{2,1}, 2T_{2,1}^2, T_{2,1}^3, T_{2,1}T_{1,1}^2, T_{2,1}^2T_{1,1}, 2T_{2,1}T_{1,1}\rangle.
	\end{gather*}
	Note that this is the Gr\"{o}bner basis, and other relations such as $T_{1,2}^3, 2T_{1,2}^2,2T_{1,2}T_{1,1}$ can be generated by the equations of Gr\"{o}bner basis.
\end{example}

In Examples~\ref{counter1}, \ref{counter2}, we can see the $\mathscr{L}$ is not a saturated ideal of first syszygies. However we provide the following trivial example to see this fact easier.
\begin{example}\label{counter3}
	We consider the polynomial ring $\mathbb{Z}/8\mathbb{Z}[x_1,x_2]$. Let $f_1=2x_1, f_2=2x_2$, and let the ideal $I$ be generated by $f_1$ and $f_2$. We consider the Rees algebra $\mathbb{Z}/8\mathbb{Z}[x_1,x_2][It]$. Then we have the epimorphism
	\begin{gather*}
	\phi:\mathbb{Z}/8\mathbb{Z}[T_{1,1},T_{2,1},x_1,x_2]\rightarrow \mathbb{Z}/8\mathbb{Z}[x_1,x_2][It]\\
	T_{k,j}\mapsto f_kt_j.
	\end{gather*}
	Suppose $\mathscr{L}=\ker(\phi)$. By~\cite[Corllary 5.4]{BJGR}, the ideal of first syzygies of $I$ is $J=\langle x_1T_{2,1}-x_2T_{1,1}, 4T_{1,1}, 4T_{2,1}\rangle$. We show that $\mathscr{L}$ is not equal to a saturated ideal of $J$. Suppose $J:h^\infty=\mathscr{L}$, where $h\in\mathbb{Z}/8\mathbb{Z}[x_1,x_2]$. If $h$ is a zero divisor, then by Corollary~\ref{nonzero-divisor1}, every term of $h$ is a multiple of 2. Hence there is a natural number $n$ such that $h^n=0$. Therefore we have
	$$
	J:h^\infty=\mathbb{Z}/8\mathbb{Z}[T_{2,1},T_{1,1},x_1,x_2].
	$$
	But we know that for example $T_{1,1}$ is not in $\mathscr{L}$. Thus in this case $J:h^\infty\neq\mathscr{L}$. So that we consider the case that $h$ is a nonzero divisor. Suppose that $J:h^\infty=\mathscr{L}$. Then $T_{1,1}^3\in J:h^\infty$. But $T_{1,1}^3\notin J$. Because if not, then 
	$$
	T_{1,1}^3=g_1(x_1T_{2,1}-x_2T_{1,1})+g_2(4T_{1,1})+g_3 (4T_{2,1}).
	$$
	We multiply both sides of equation above by 2 and we get
	$$
	2T_{1,1}^3=2g_1(x_1T_{2,1}-x_2T_{1,1}),
	$$
	which is not possible. Because all terms on the right side of equation have either $x_1$ or $x_2$. Then there is a natural number $n$ such that $h^nT_{1,1}^3\in J$. But by Corollary~\ref{nonzero-divisor1}, at least one of terms of $h^n$ is not a multiple of 2. Now we have
	$$
	h^nT_{1,1}^3=g_1(x_1T_{2,1}-x_2T_{1,1})+g_2(4T_{1,1})+g_3 (4T_{2,1}).
	$$ 
	Again we multiply both sides of the equation above by 2 and we get
	$$
	2h^nT_{1,1}^3=2g_1(x_1T_{2,1}-x_2T_{1,1}).
	$$
	We have $0\neq2h^nT_{1,1}^3=b_1T_{1,1}^3+\dots+b_kT_{1,1}^3$, where $b_i$ are terms that only involve some constants and $x$'s. On the other hand $2g_1=a_1+\dots+a_m$, where $a_i$ are terms. Then we have
	$$
	b_1T_{1,1}^3+\dots+b_kT_{1,1}^3=a_1x_1T_{2,1}+\dots+a_mx_1T_{2,1}-a_1x_2T_{1,1}-\dots-a_mx_2T_{1,1}.
	$$
	On the right side of the equation every $a_ix_1T_{2,1}$ must cancel. On the other hand every $a_ix_1T_{2,1}$ can only cancel with an $a_jx_2T_{1,1}$. But since right side is not zero, this not possible. Which is a contradiction.
\end{example}

\section{Defining ideal of the multi-Rees algebra when the ground ring is PID}\label{pid-section}

In this section we fix ideals $I_1,\dots,I_r$ of $R[\mathbf{x}]$, where $\mathbf{x}=x_1,\dots,x_n$, and $R$ is a PID. We denote a fixed generating set of the ideal $I_j$ by $\gen(I_j)$. Let $\{f_1,\dots,f_q\}=\bigcup_{i=1}^{r}\gen(I_i)$. Let 
$$
S:=R[\mathbf{x}]\left[\{T_{k,j}\}_{1\le j\le r,f_k\in\gen(I_j)}\right].
$$ 
We define and fix the $R[\mathbf{x}]$-algebra epimorphism 
$$
\phi : S \to R[\mathbf{x}][I_1t_1,\dots, I_rt_r], \ \text{by} \ \phi(T_{k,j})=f_kt_j,
$$
where $R[\mathbf{x}][I_1t_1,\dots, I_rt_r]$ is the multi-Rees algebra of ideals $I_j$. Let $\mathscr{L}=\ker(\phi)$. In the rest of this section $\mathbf{T}$ is the set of defined $T_{k,j}$. In this section, we fix the lexicographic order $\prec_2$ in $R[\mathbf{x},\mathbf{T}]$ on $T$ monomials. That means $T_{k_1,j_1}\prec_2 T_{k_2,j_2}$ if either $k_1<k_2$ or if $k_1=k_2$, then $j_1<j_2$. We denote the leading term of $f$ by $\LT(f)$. We also denote the leading monomial of $f$ (resp. the leading coefficient of $f$) by $\LM(f)$ (resp. $\LC(f)$). Note that when we put terms of a polynomial in descending order, coefficients are polynomials in $x$. Since we have not put any order on $x$, then leading coefficient of $f$ is a polynomial in $x$, and the leading monomial of $f$ is a monomial in $T$. 

\begin{theorem}\label{division-algorithm-1}
Let $F=(f_1,\dots,f_t)$ be an ordered $t$-tuple of polynomials in $R[\mathbf{x},\mathbf{T}]$ and let $f$ be a polynomial in $R[\mathbf{x},\mathbf{T}]$. Consider the following algorithm:
	
1. $g_1=g_2=\dots=g_t=0$, $s=0$, $g=f$, and $h=f$.
	
2. Find the smallest $j\in\{1,\dots,t\}$ such that $\LM(g)$ is a multiple of $\LM(f_j)$. If such a $j$ exists, replace $g_j$ by $\LC(f_j)\left(g_j+\frac{\LT(g)}{\LT(f_j)}\right)$, replace the other $g_k$ by $\LC(f_j)g_k$. Finally, replace $g$ by 
$$
\LC(f_j)\left(g-\frac{\LT(g)}{\LT(f_j)}f_j\right),
$$ 
and $s$ by $\LC(f_j)s$. Also, replace $h$ by $\LC(f_j)h$.
	
3. Repeat step 2 until there is no more $j\in\{1,\dots,t\}$ such that $\LM(g)$ is a multiple of $\LM(f_j)$. Then replace $s$ by $s+\LT(g)$ and $g$ by $g-\LT(g)$.
	
4. If now $g\neq 0$, start again with step 2. If $g=0$ stop.
	
This is an algorithm which gives us $t$-tuple $(g_1,\dots,g_t)\in R[\mathbf{x},\mathbf{T}]^t$ and the polynomial $s\in R[\mathbf{x},\mathbf{T}]$ such that
$$
af=h=g_1f_1+\dots+g_tf_t+s,
$$
and such that the following conditions are satisfied.
	
a. If $s\neq 0$, then none of terms of $s$ is divisible by any of $\LM(f_1),\dots,\LM(f_t)$.
	
b. $a$ is a product of $\LC(f_j)$.  
\end{theorem}
The algorithm terminates in finitely many steps, because in each step the leading term becomes strictly smaller

In this section by $\langle\dots\rangle$ we mean the ideal generated by elements inside the brackets. In particular for a set $A$, by $\langle A\rangle$, we mean the ideal generated by all elements of the set $A$.

Now we want to show that the defining ideal (i.e. $\ker(\phi)=\mathscr{L}$) of the multi-Rees algebra $R[\mathbf{x}][I_1t_1,\dots, I_rt_r]$ is a saturated ideal. For every $1\le j\le r$, let $T_{k_j,j}$ be the smallest $T$ variable under $\prec_2$. Then, in this section, we fix the ideal 
\begin{gather*}
\mathcal{L}=\langle\{f_{k_j}T_{k,j}-f_kT_{k_j,j};1\le j\le r,f_k\in I_j\}\rangle.
\end{gather*}

Using Theorem~\ref{division-algorithm-1}, and with a similar idea in the proof of Theorem~\ref{equations2}, we can prove the following result.
\begin{theorem}\label{equations3}
If
$$
\mathcal{E}=\mathcal{L}:\left(\prod f_{k_j}\right)^\infty,
$$
then $\mathscr{L}=\mathcal{E}$.
\end{theorem}

\begin{proposition}\label{equations-grobner2}
Let 
$$
\mathcal{F}=\langle y\prod f_{k_j}-1\rangle+\mathcal{L}\subseteq R[\mathbf{x},\mathbf{T},y],
$$
and
$$
\mathcal{E}=\mathcal{L}:\left(\prod f_{k_j}\right)^\infty.
$$
Then $\mathcal{E}=\mathcal{F}\cap R[\mathbf{x},\mathbf{T}]$.
\end{proposition}
\begin{proof}
The proof is similar to the proof of Proposition~\ref{equations-grobner1}.
\end{proof}
\begin{corollary}
Suppose
$$
\mathcal{F}=\langle y\prod f_{k_j}-1\rangle+\mathcal{L}\subseteq R[\mathbf{x},\mathbf{T},y].
$$ 
Let $\prec$ be an elimination order for $y$ on the ring $R[\mathbf{x},\mathbf{T},y]$ and let $G$ be a Gr\"{o}bner basis for $\mathcal{F}$ with this order. Then $G\cap R[\mathbf{x},\mathbf{T}]$ is a Gr\"{o}bner basis for $\mathscr{L}$.
\end{corollary}

\begin{corollary}
Let all ideals $I_j$ be generated by terms. Suppose 
$$
\mathcal{E}=\mathcal{L}:(p_1 ... p_s x_1 ... x_n)^\infty,
$$ 
where $p_i$ are all non-associate prime factors of $f_{k_j}$'s. Then $\mathscr{L}=\mathcal{E}$.
\end{corollary}

Now we provide an example. In the following example, whenever we compute Gr\"{o}bner basis, we use the following algorithms in~\cite{gamanda2019syzygy}: Division Algorithm 2.3, $S$-Polynomial Algorithm 2.4, and Buchberger's Algorithm 2.5. We just remind that the meaning of the ideal $\mathcal{F}$, is mentioned in Proposition~\ref{equations-grobner2}.

\begin{example}
We consider the polynomial ring $\mathbb{Z}[x_1,x_2,x_3]$. Let $f_1=6x_1^2x_2, f_2=3x_1x_3, f_3=5x_1x_3^2, f_4=x_2x_3$. We consider the ideals $I_1=\langle f_1,f_2,f_3\rangle$ and $I_2=\langle f_1,f_2,f_4\rangle$. Then we have the epimorphism
\begin{gather*}
\phi:\mathbb{Z}[T_{4,2},T_{2,2},T_{1,2},T_{3,1},T_{2,1},T_{1,1},x_1,x_2,x_3]\rightarrow\mathbb{Z}[x_1,x_2,x_3][I_1t_1,I_2t_2]\\
T_{k,j}\mapsto f_kt_j.
\end{gather*}
Suppose that $\mathscr{L}=\ker(\phi)$. We have
\begin{gather*}
\mathcal{F}=\langle 2.3yx_1x_2x_3-1, 6x_1^2x_2T_{2,1}-3x_1x_3T_{1,1}, 6x_1^2x_2T_{3,1}-5x_1x_3^2T_{1,1},\\
6x_1^2x_2T_{2,2}-3x_1x_3T_{1,2}, 6x_1^2x_2T_{4,2}-x_2x_3T_{1,2}\rangle.
\end{gather*}
We put lexicographic order on the ring
$$
\mathbb{Z}[y,T_{4,2},T_{2,2},T_{1,2},T_{3,1},T_{2,1},T_{1,1},x_1,x_2,x_3].
$$
Let $G$ be intersection of Gr\"{o}bner basis of $\mathcal{F}$ with 
$$
\mathbb{Z}[T_{4,2},T_{2,2},T_{1,2},T_{3,1},T_{2,1},T_{1,1},x_1,x_2,x_3].
$$
Then we have
\begin{gather*}
\mathscr{L}=G=\langle 2x_1x_2T_{2,1}-x_3T_{1,1}, 3T_{3,1}-5x_3T_{2,1}, 2x_1x_2T_{2,2}-x_3T_{1,2},\\
T_{2,2}T_{1,1}-T_{1,2}T_{2,1}, 3x_1T_{4,2}-x_2T_{2,2}, 3x_3T_{4,2}T_{1,1}-2x_2^2T_{2,2}T_{2,1},\\
5x_3^2T_{4,2}T_{1,1}-2x_2^2T_{2,2}T_{3,1}, 5x_1x_3T_{4,2}T_{1,1}-x_2T_{1,2}T_{3,1},5x_1x_3T_{4,2}T_{2,1}-x_2T_{2,2}T_{3,1},\\
10x_1^2T_{4,2}T_{2,1}-T_{1,2}T_{3,1}, 3x_3T_{4,2}T_{1,2}-2x_2^2T_{2,2}^2, 5x_3^3T_{4,2}T_{1,2}T_{2,1}-2x_2^2T_{2,2}^2T_{3,1}\rangle.
\end{gather*}
\end{example}

\end{document}